\documentclass{amsart}

\usepackage[utf8]{inputenc} 
\usepackage[T1]{fontenc} 
\usepackage[english]{babel}
\usepackage{amsmath}
\usepackage{amssymb}

\theoremstyle{plain} 
\newtheorem{theorem}{Theorem} 
\newtheorem{lemma}[theorem]{Lemma} 
\newtheorem{corollary}[theorem]{Corollary} 

\theoremstyle{remark} 
 
\newtheorem*{remark}{Remark}

\DeclareMathOperator{\supp}{supp}

\DeclareMathOperator{\mre}{Re}
\begin{document} 

\title[Multiple Fourier series and capacity]{Rectangular summation of multiple Fourier series and multi-parametric capacity}
\date{\today} 

\author{Karl-Mikael Perfekt}
\address{Department of Mathematics and Statistics,
	University of Reading, Reading RG6 6AX, United Kingdom}
\email{k.perfekt@reading.ac.uk}

\begin{abstract}
We consider the class of multiple Fourier series associated with functions in the Dirichlet space of the polydisc. We prove that every such series is summable with respect to unrestricted rectangular partial sums, everywhere except for a set of zero multi-parametric logarithmic capacity. Conversely, given a compact set in the torus of zero capacity, we construct a Fourier series in the class which diverges on this set, in the sense of Pringsheim. We also prove that the multi-parametric logarithmic capacity characterizes the exceptional sets for the radial variation and radial limits of Dirichlet space functions. As a by-product of the methods of proof, the results also hold in the vector-valued setting.
\end{abstract}

\subjclass[2010]{}

\maketitle
\section{Introduction}
This article will consider unrestricted rectangular summation and other multi-parameter summation methods of the multiple Fourier series
\begin{equation} \label{eq:fintro}
f(\theta) \sim \sum_{\alpha \in \mathbb{Z}^n} a_\alpha e^{i(\alpha_1\theta_1 + \cdots \alpha_n \theta_n)}.
\end{equation}
To clarify this objective, note that there are several natural ways to form the partial sums of a multiple Fourier series. For example, one can attempt to sum the series via \textit{square partial sums},
$$ \lim_{M \to \infty}  \sum_{|\alpha_j| \leq M} a_\alpha e^{i(\alpha_1\theta_1 + \cdots \alpha_n \theta_n)},$$
\textit{spherical partial sums},
$$\lim_{R \to \infty} \sum_{\alpha_1^2 + \cdots + \alpha_n^2 \leq R} a_\alpha e^{i(\alpha_1\theta_1 + \cdots \alpha_n \theta_n)},$$
or \textit{unrestricted rectangular partial sums},
\begin{equation} \label{eq:introrectsum}
\lim_{\mathbb{N}^n \ni N \to \infty} \sum_{|\alpha_j| \leq N_j} a_\alpha e^{i(\alpha_1\theta_1 + \cdots \alpha_n \theta_n)},
\end{equation}
where $N \to \infty$ means that $\min_{1 \leq j \leq n} N_j \to \infty$, with no assumption made on the relationship between $N_j$ and $N_k$, $1 \leq j, k \leq n$. These three modes of convergence behave quite differently, and typically require different techniques to treat. The first two summation methods only depend on one parameter ($M$ or $R$), while the the third is an example of a multi-parameter summation method. We refer to \cite{AshWelland} and \cite[Ch. XVII]{ZygmundBook} for an introduction to multi-parameter summation methods for Fourier series.

Carleson \cite{Carleson} famously proved that the Fourier series of a function $f \in L^2(\mathbb{T})$ converges for almost every $\theta \in [0, 2\pi)$. This can be exploited to show that the Fourier series of a function $f \in L^2(\mathbb{T}^n)$, $n \geq 2$, converges with respect to square partial sums for almost every $\theta \in [0,2\pi)^n$ \cite{Antonov, Feff2, Sjolin, Tevzadze}. On the other hand, C. Fefferman \cite{Feff} constructed a continuous function $f \in C(\mathbb{T}^2)$ whose Fourier series diverges with respect to unrestricted rectangular sums for every $\theta \in [0,2\pi)^2$. Under spherical summation, the convergence question is still open for Fourier series of $f \in L^2(\mathbb{T}^n)$, $n \geq 2$, but we refer to \cite{Karagulyan} for some related negative results.

Let us now bring potential theory into the discussion. For a series $f(\theta) \sim \sum_{k\in \mathbb{Z}} a_k e^{ik\theta}$ such that $\sum_{k \in \mathbb{Z}} |k| |a_k|^2 < \infty$, Beurling \cite{Beurling} showed that $f(\theta)$ is summable for every $\theta \in \mathbb{T} \setminus E$, where $E$ is a set of zero logarithmic capacity. This was given a one-parameter generalization to multiple Fourier series by Lippman and Shapiro \cite{LippShap74}. They proved that if $f \in L^1(\mathbb{T}^n)$, $n\geq 2$, is as in \eqref{eq:fintro} and satisfies that $\sum_{\alpha \in \mathbb{Z}^n} (\alpha_1^2 + \cdots + \alpha_n^2) |a_\alpha|^2 < \infty$, then $f(\theta)$ is summable with respect to spherical partial sums, except for on a set $E \subset \mathbb{T}^n$ of zero ordinary capacity (logarithmic capacity for $n=2$ and Newtonian capacity for $n \geq 3$, under the identification $\mathbb{T}^n \simeq (\mathbb{R}/\mathbb{Z})^n$).

An interest in the multi-parameter summation method \eqref{eq:introrectsum} thus leads us to seek a suitable concept of capacity. A notion of multi-parametric logarithmic capacity has appeared recently in function-theoretic investigations of the Dirichlet space $\mathcal{D}(\mathbb{D}^n)$ of the polydisc \cite{BCLS15, BKKLSS16, Bergqvist, Kaptan}. In particular, in \cite{AMPS}, it was proven that bi-parameter logarithmic capacity characterizes the Carleson measures of $\mathcal{D}(\mathbb{D}^2)$. It is therefore natural to generalize Beurling's result to this context.

Before stating the main results, let us fix some notation. For a positive integer $n$, consider the multiple Fourier series
$$f(\theta) \sim \sum_{\alpha \in \mathbb{N}^n} a_\alpha e^{i(\alpha, \theta)},$$
where $\mathbb{N} = \{0, 1, 2, \ldots\}$, $\theta \in [0, 2\pi)^n$, and the coefficients belong to some Hilbert space $\mathcal{H}$, $a_\alpha \in \mathcal{H}$. We say that $f$ belongs to the Dirichlet space of the $n$-disc, $f \in \mathcal{D}(\mathbb{D}^n, \mathcal{H})$, if
$$\sum_{\alpha \in \mathbb{N}^n} (\alpha_1+1) \cdots (\alpha_n+1) \|a_\alpha\|_{\mathcal{H}}^2 < \infty.$$
If $\mathcal{H} = \mathbb{C}$, we simply write $\mathcal{D}(\mathbb{D}^n)$. Occasionally, it will be very useful for us to view for example the Dirichlet space of the bidisc as a Dirichlet space-valued one-variable Dirichlet space,
$$\mathcal{D}(\mathbb{D}^2) = \mathcal{D}(\mathbb{D}, \mathcal{D}(\mathbb{D})).$$
This is the reason that we consider the vector-valued setting.

Through iterated Poisson extension, any $f \in \mathcal{D}(\mathbb{D}^n, \mathcal{H})$ defines an $\mathcal{H}$-valued holomorphic function in $z = (r_1 e^{i\theta_1}, \ldots, r_n e^{i\theta_n}) \in \mathbb{D}^n$, 
$$f(z) = f_r(\theta) =  \sum_{\alpha \in \mathbb{N}^n} a_\alpha r^\alpha e^{i(\alpha, \theta)}, \quad r \in [0,1)^n, \; \theta \in [0, 2\pi)^n.$$
We will freely identify $[0, 2\pi)^n$ with the $n$-torus $\mathbb{T}^n$.

For a positive measurable function $f$ on $\mathbb{T}^n$, let
$$Bf(\theta) = \int_{\mathbb{T}^n} \frac{1}{|e^{i\theta_1} - e^{i\psi_1}|^{\frac{1}{2}}} \cdots \frac{1}{|e^{i\theta_n} - e^{i\psi_n}|^{\frac{1}{2}}}  f(\psi) \, d\psi,$$
where $d\psi$ denotes the normalized Lebesgue measure on $\mathbb{T}^n$. For a set $E \subset \mathbb{T}^n$ in the $n$-torus, we then define the following outer capacity:
\begin{equation} \label{eq:outercap}
C(E) = \inf \left \{ \|f\|^2_{L^2(\mathbb{T}^n)} \, : \, f \geq 0, \; Bf(\theta) \geq 1 \textrm{ for all }  \theta \in E \right \}.
\end{equation}
When $n=1$ and $E$ is a Borel set (or more generally a capacitable set, see Section~\ref{sec:prelim}), $C(E)$ is equivalent to the usual (gently modified) logarithmic capacity of $E$. For $n \geq 2$, $C(E)$ is a multi-parameter analogue of logarithmic capacity. The capacity $C(\cdot)$ fits the general theory of \cite[Ch. 2.3--2.5]{Hedberg}, allowing us to access certain basic tools of potential theory such as equilibrium measures. However, we warn the reader that a number of familiar properties from the one-parameter setting do not hold. Notably, the associated $n$-logarithmic potentials defined in Section~\ref{sec:prelim} generally fail to satisfy any kind of boundedness principle \cite{AMPS}.

We shall actually prove convergence in a stronger sense than that given by \eqref{eq:introrectsum}. We say that the series $f(\theta)$ \textit{converges in the sense of Pringsheim} if it converges with respect to unrestricted rectangular partial sums,
\begin{equation} \label{eq:pring1}
f(\theta) = \lim_{\mathbb{N}^n \ni N \to \infty} \sum_{\alpha_1= 0}^{N_1} \cdots \sum_{\alpha_n= 0}^{N_n} a_\alpha e^{i(\alpha, \theta)},\end{equation}
and it holds that
\begin{equation} \sup_{N \in \mathbb{N}^n} \left\|\sum_{\alpha_1= 0}^{N_1} \cdots \sum_{\alpha_n= 0}^{N_n} a_\alpha e^{i(\alpha, \theta)}\right\|_{\mathcal{H}} < \infty.\end{equation}

Finally, we say that a property holds quasi-everywhere if it holds everywhere on $\mathbb{T}^n$ but for a set of capacity 0. Our first main result is the following.
\begin{theorem} \label{thm:rectconv}
If $f \in \mathcal{D}(\mathbb{D}^n, \mathcal{H})$, then for quasi-every $\theta \in [0,2\pi)^n$, $f(\theta)$ converges in the sense of Pringsheim.
\end{theorem}
Our second main theorem shows that Theorem~\ref{thm:rectconv} is sharp.
\begin{theorem} \label{thm:rectsharp}
If $E \subset \mathbb{T}^n$ is compact and $C(E)=0$, then there exists a function $f \in \mathcal{D}(\mathbb{D}^n)$ such that $f(\theta)$ diverges in the sense of Pringsheim for $\theta \in E$.
\end{theorem}
To prove Theorems~\ref{thm:rectconv} and \ref{thm:rectsharp}, we will first prove that multi-parametric logarithmic capacity characterizes the exceptional sets for the radial variation $V_n f (\theta)$ of $f \in \mathcal{D}(\mathbb{D}^n, \mathcal{H})$,
$$V_n f (\theta) = \int_{[0,1]^n} \| \partial_{r} f_r(\theta) \|_{\mathcal{H}} \, dr, $$
where $\partial_r = \partial_{r_1} \cdots \partial_{r_n}$ and $dr = dr_1 \cdots dr_n$. 
\begin{theorem} \label{thm:beurling}
If $f \in \mathcal{D}(\mathbb{D}^n, \mathcal{H})$, then $V_n f(\theta)$ is finite for quasi-every $\theta$.
\end{theorem}
\begin{remark}
When $n=2$ and $\mathcal{H} = \mathbb{C}$, this theorem is an immediate corollary of the work in \cite{AMPS}. In that paper, the Carleson measures for $\mathcal{D}(\mathbb{D}^2)$, which also turn out to be embedding measures for the radial variation, were given a potential-theoretic characterization. However, the characterization of Carleson measures is a much more complicated problem than the characterization of exceptional sets for the radial variation -- see \cite{HPV, MPV}.
\end{remark}
Applying Theorem~\ref{thm:beurling}, we obtain the following corollary on unrestricted iterated Abel summation, that is, on the radial limits of a function $f \in \mathcal{D}(\mathbb{D}^n, \mathcal{H})$.
\begin{corollary} \label{cor:conv}
If $f \in \mathcal{D}(\mathbb{D}^n, \mathcal{H})$, then for quasi-every $\theta$ it holds that
$$f^\ast(\theta) = \lim_{r \to (1,\cdots, 1)} f_r(\theta)$$
exists, and furthermore that
$$\sup_{r} \|f_r(\theta)\|_{\mathcal{H}} < \infty.$$
The value of $f^\ast(\theta)$ coincides with the Pringsheim sum $f(\theta)$ quasi-everywhere.
\end{corollary}
Theorem~\ref{thm:beurling} is also sharp.
\begin{theorem} \label{thm:sharp}
If $E \subset \mathbb{T}^n$ is compact and $C(E)=0$, then there exists a function $f \in \mathcal{D}(\mathbb{D}^n)$ such that
$$\lim_{z \to \zeta} \mre f(z) = \infty, \quad \zeta \in E.$$
\end{theorem}

To complete the analogy with Beurling's work \cite{Beurling}, we shall also prove the following result on the strong differentiability of the integral of $f$.
For $\theta \in [0,2\pi)^n$ and $h \in (0,\pi)^n$, let
$$F_h(\theta) = \frac{\pi^n}{h_1\cdots h_n} \int_{(\theta_1 - h_1, \theta_1 + h_1)} \cdots \int_{(\theta_n - h_n, \theta_n + h_n)} f(\psi) \, d\psi.$$
\begin{theorem} \label{thm:strongdiff}
If $f \in \mathcal{D}(\mathbb{D}^n, \mathcal{H})$, then
$$\lim_{h \to (0,\ldots, 0)} F_h(\theta) = f(\theta)$$
for quasi-every $\theta$.
\end{theorem}

\noindent \textbf{Acknowledgments.} The author is grateful to the anonymous referee for their suggestions, which helped to improve the exposition. This research was partially supported by EPSRC grant EP/S029486/1.

\section{Preliminaries} \label{sec:prelim}
\subsection{Multi-parametric capacity}
First, let us slightly modify the kernel of $B$ (without otherwise changing the notation). Letting
$$b(\theta) = 3 + \sum_{k=1}^\infty \frac{\cos k\theta}{k^{\frac{1}{2}}}, \quad \theta \in [0, 2\pi),$$
we note that $b(\theta) \geq 1$ is convergent and continuous for $\theta > 0$, and that 
$$b(\theta) \approx \left| \sin \frac{\theta}{2}\right|^{-\frac{1}{2}}.$$
See \cite[Ch. V.1--V.2]{ZygmundBook}.
Hence, if we let $B(\theta) = b(\theta_1) \cdots b(\theta_n)$, and for positive finite Borel measures $\mu$ on $\mathbb{T}^n$ define
$$B\mu(\theta) = \int_{\mathbb{T}^n} B(\theta - \psi) \, d\mu(\psi), \quad \theta \in [0, 2\pi)^n,$$
this only changes the definition of $C(\cdot)$ in \eqref{eq:outercap} up to constants. 

Note that the convolution of $b$ with itself satisfies that
$$h(\theta) := b*b(\theta) = 9 + \frac{1}{2}\log \frac{1}{|1-e^{i\theta}|}.$$
The kernel $H(\theta) = h(\theta_1) \cdots h(\theta_n)$ defines the $n$-logarithmic potential,
$$H\mu(\theta) =  \int_{\mathbb{T}^n} H(\theta - \psi) \, d\mu(\psi), \quad \theta \in [0, 2\pi)^n.$$
The energy of a measure $\mu$ is thus given by
$$\|B\mu\|_{L^2(\mathbb{T}^n)}^2 = \int_{\mathbb{T}^n} H\mu(\theta) \, d\mu(\theta) = \int_{\mathbb{T}^n} \int_{\mathbb{T}^n} H(\theta - \psi) \, d\mu(\psi) \, d\mu(\theta).$$

Since $B(\theta)$ is lower semi-continuous on $\mathbb{T}^n$, the theory of \cite[Ch. 2.3--2.5]{Hedberg} applies to $C(\cdot)$, as was mentioned in the introduction. In particular, every Borel set $E \subset \mathbb{T}^n$ is capacitable, that is,
$$C(E) = \inf \{ C(G) \, : \, G \supset E \textrm{ open}\} = \sup \{ C(K) \, : \, K \subset E \textrm{ compact}\}.$$
For any capacitable set $E$, $C(E)$ can be computed through the dual definition of capacity, which might give the reader a more familiar definition in the case of logarithmic capacity. More precisely,
\begin{equation} \label{eq:dualdef}
C(E)^{1/2} = \sup \left\{ \mu(E) \, : \, \supp \mu \subset E, \; \|B\mu\|_{L^2(\mathbb{T}^n)} \leq 1 \right \}.
\end{equation}
In particular, the set $E$ has capacity $0$, $C(E) = 0$, if and only if every non-zero positive finite measure $\mu$ with support in $E$ has infinite energy,
$$\int_{\mathbb{T}^n} \int_{\mathbb{T}^n} H(\theta - \psi) \, d\mu(\psi) \, d\mu(\theta) = \infty.$$
Furthermore, the following simple lemma, which we shall use without mention, is clear from \eqref{eq:outercap} and \eqref{eq:dualdef}.
\begin{lemma}
If $E_1, \ldots, E_n$ are Borel sets, then
$$C(E_1 \times \cdots \times E_n) = C(E_1)\cdots C(E_n).$$
\end{lemma}
The final piece of information that we require is the existence of equilibrium measures. For any compact set $K \subset \mathbb{T}^n$, the extremal to the capacity problem is generated by a measure $\mu_K$ such that: $\supp \mu_K \subset K$, $H\mu_K(\theta) \leq 1$ for $\theta \in \supp \mu_K$, $H\mu(\theta) \geq 1$ for quasi-every $\theta \in K$ and
$$\mu_K(K) = \int_{\mathbb{T}^n} \int_{\mathbb{T}^n} H(\theta - \psi) \, d\mu(\psi) \, d\mu(\theta) = C(K).$$
\subsection{$n$-harmonic functions}
A continuous function on $\mathbb{D}^n$ is $n$-harmonic if it is harmonic in each variable $z_j$ separately, $z = (z_1, \ldots, z_n) \in \mathbb{D}^n$.  For a finite measure $\mu$ on $\mathbb{T}^n$, we denote by $P\mu$ the $n$-harmonic function
$$P\mu(z) = P\mu(r, \theta) = \int_{\mathbb{T}^n} P_{r_1} (\theta_1 - \psi_1) \cdots P_{r_n}(\theta_n - \psi_n) \, d\mu(\psi),$$
where $z = (r_1e^{i\theta_1}, \ldots, r_n e^{i\theta_n}) \in \mathbb{D}^n$ and $P_r(\theta)$ denotes the usual Poisson kernel,
$$P_r(\theta) = \frac{1-r^2}{1-2r\cos \theta + r^2}.$$
We refer to \cite[Ch. 2]{Rudin} for the fundamentals of $n$-harmonic functions and multiple Poisson integrals. We only need to know the following, which can be extracted from Theorems 2.1.3 and 2.3.1 in \cite{Rudin}.
\begin{lemma} \label{lem:poissonrep}
If $u \geq 0$ is $n$-harmonic and non-negative on $\mathbb{D}^n$, then there exists a function $0 \leq g \in L^1(\mathbb{T}^n)$ and a singular measure $\sigma \geq 0$ on $\mathbb{T}^n$ such that
$$u(z)  = P\nu(z), \quad d\nu = g \,d\theta + d\sigma, \; z \in \mathbb{D}^n.$$
Furthermore, for almost every $\theta \in [0,2\pi)^n$, it holds that
$$\lim_{t \to 1^{-}} u(te^{i\theta_1}, \ldots te^{i\theta_n}) = g(\theta).$$
\end{lemma}
\begin{remark}
Since we will prove theorems about unrestricted summation and strong differentiability, we note that unlike the one-variable setting, the proof of the lemma does not specify for which points $\theta$ the limit exists. In general, localization fails for multiple Poisson integrals. In fact, let $f^1 \in C^\infty(\mathbb{T})$ be such that $f^1(\theta_1) = 0$ for $|\theta_1| \leq \varepsilon$, for some $\varepsilon > 0$, and such that there is a sequence $t_j \to 1$ for which $P[f^1\, d\theta_1] (t_j, 0) > 0$. Let $f^2 \in \mathcal{D}(\mathbb{D})$ be any function such that $\lim_{t\to1} \mre P[f^2 \, d\theta_2](t, 0) = \infty$. Let 
$$f(\theta) = f^1(\theta_1) f^2(\theta_2) \sim \sum_{\alpha \in \mathbb{Z}^2} a_\alpha e^{i(\alpha, \theta)}.$$
Then the Fourier coefficients of $f$ satisfy that 
$$\sum_{\alpha \in \mathbb{Z}^2} (|\alpha_1|+1)(|\alpha_2|+1) |a_\alpha|^2 < \infty,$$
and $f(\theta)$ vanishes in an open neighborhood of $0$, but still 
$$\lim_{(r_1, r_2) \to (1,1)} P[f \, d\theta](r, 0) \neq 0.$$
In fact, the limit does not exist.
\end{remark}
\section{Convergence theorems}
We begin by proving Theorem~\ref{thm:beurling}. Given $f \in \mathcal{D}(\mathbb{D}^n, \mathcal{H})$, note that 
\begin{equation} \label{eq:excepset}
E = \left \{\theta \, : \, V_n f(\theta) = \infty \right \} = \bigcap_{i \geq 1} \bigcup_{j \geq 1} \left \{\theta \, : \, \int_{[0, 1-1/j]^n} \|\partial_r f_r(\theta)\|_{\mathcal{H}} \, dr > i  \right\}
\end{equation}
is a $G_\delta$-set, hence capacitable. The following proof is in the spirit of Salem and Zygmund's approach to exceptional sets for one-variable Dirichlet spaces \cite{SalemZygmund}.
\begin{proof}[Proof of Theorem~\ref{thm:beurling}] 
We may assume that the Fourier coefficients of $f$ are supported in $(\mathbb{Z}_{\geq 1})^n$, $f \sim \sum_{\alpha \in (\mathbb{Z}_{\geq 1})^n} a_\alpha e^{i(\alpha, \theta)}$.
For $k \geq 0$, let
\begin{equation} \label{eq:binomasymp}
c_k = \binom{k-1/2}{k} = \frac{1}{\sqrt{\pi}k^{1/2}} \left( 1 + O(k^{-1}) \right),
\end{equation}
so that 
$$\tilde{b}(\theta) := \sum_{k=0}^\infty c_k \cos k\theta = \mre \frac{1}{(1-e^{i\theta})^{1/2}}, \quad 0 < \theta < 2\pi,$$
see \cite[Ch. V.2]{ZygmundBook}. Note that $\tilde{b}(\theta)$ is another uniformly positive function with the same singular behavior as $b(\theta)$. Let $\tilde{h} = \tilde{b} * \tilde{b}$. Then $\tilde{h} \geq c > 0$ for some $c$, and by \eqref{eq:binomasymp} we see that $\tilde{h}(\theta)$ has the same logarithmic singularity as $h(\theta)$, when $\sin \frac{\theta}{2} \to 0$. Let $\tilde{B}(\theta) = \tilde{b}(\theta_1)\cdots \tilde{b}(\theta_n)$, $\tilde{H}(\theta) = \tilde{h}(\theta_1)\cdots \tilde{h}(\theta_n)$, and for $r \in [0,1)^n$,
$$\tilde{B}_r(\theta) = P[\tilde{B}(\psi) \, d\psi](r, \theta) =: \sum_{\alpha \in \mathbb{Z}^n} C_\alpha r_1^{|\alpha_1|} \cdots r_n^{|\alpha_n|} e^{i(\alpha, \theta)}.$$
Note that
\begin{equation} \label{eq:Bcoeff}
C_\alpha = \frac{c_{\alpha_1} \cdots c_{\alpha_n}}{2^n}, \quad \alpha \in (\mathbb{Z}_{\geq 1})^n.
\end{equation}
We will also rely on the estimate
\begin{align} \label{eq:Brest}
\int_{[0,1]^n} |\partial_r \tilde{B}_r(\theta)| \, dr &\lesssim \int_{[0,1]^n} \frac{1}{|1-r_1e^{i\theta_1}|^{3/2}} \cdots \frac{1}{|1-r_n e^{i\theta_n}|^{3/2}} \, dr \\ &\lesssim \left| \sin \frac{\theta_1}{2}\right|^{-\frac{1}{2}} \cdots \left| \sin \frac{\theta_n}{2}\right|^{-\frac{1}{2}} \lesssim \tilde{B}(\theta). \notag
\end{align}

Suppose now that the set $E$ of \eqref{eq:excepset} has positive capacity. Then there exists a non-zero finite measure $\mu$, supported in $E$, such that
$$\|\tilde{B}\mu \|_{L^2(\mathbb{T}^n)}^2 = \int_{\mathbb{T}^n} \int_{\mathbb{T}^n} \tilde{H}(\theta - \psi) \, d\mu(\psi) \, d\mu(\theta) < \infty,$$
where $\tilde{B}\mu(\theta) = \int_{\mathbb{T}^n} \tilde{B}(\theta - \psi) \, d\mu(\psi)$. Let $F$ be the $\mathcal{H}$-valued series
$$F(\theta) \sim \sum_{\alpha \in (\mathbb{Z}_{\geq 1})^n} C_\alpha^{-1} a_\alpha e^{i(\alpha, \theta)}.$$
The coefficients of $F$ are square-summable, by \eqref{eq:binomasymp}, \eqref{eq:Bcoeff}, and the fact that $f \in \mathcal{D}(\mathbb{D}^n, \mathcal{H})$. Thus $F(\theta)$ has meaning for almost every $\theta$, and 
$$\int_{\mathbb{T}^n} \|F(\theta)\|_{\mathcal{H}}^2 \, d\theta < \infty.$$
By our assumption on the support of the Fourier coefficients of $f$ we have that
$$\partial_r f_r(\theta) = \int_{\mathbb{T}^n} F(\psi) \partial_r \tilde{B}_r(\theta - \psi) \, d\psi,$$
and therefore by \eqref{eq:Brest} that
$$V_nf(\theta) \lesssim \int_{\mathbb{T}^n} \|F(\psi)\|_{\mathcal{H}} \tilde{B}(\theta - \psi) \, d\psi.$$
But then, by the assumption of finite energy,
\begin{align*}
\left(\int_{\mathbb{T}^n} V_n f(\theta) \, d\mu(\theta)\right)^2 &\lesssim \left(\int_{\mathbb{T}^n} \|F(\psi)\|_{\mathcal{H}} \tilde{B}\mu (\psi) \, d\psi\right)^2 \\ &\leq \|\tilde{B}\mu \|_{L^2(\mathbb{T}^n)}^2\int_{\mathbb{T}^n} \|F(\psi)\|_{\mathcal{H}}^2 \, d\psi < \infty.
\end{align*}
This is obviously a contradiction.
\end{proof}
\begin{proof}[Proof of Corollary~\ref{cor:conv}]
We give the proof for $n=2$. The proof is the same for $n \geq 3$, but the notation is more difficult. Given $f \in \mathcal{D}(\mathbb{D}^2, \mathcal{H})$, define $f^1, f^2 \in \mathcal{D}(\mathbb{D}, \mathcal{H})$ by 
$$f^1(z) = f(z,0), \; f^2(0,w) = f(0, w), \quad z,w \in \mathbb{D}.$$
Let
$$E = \left \{\theta \in [0,2\pi)^2 \, : \, V_2 f(\theta) = \infty \right \},$$
and
$$E_1 = \left \{\theta_1 \in [0,2\pi) \, : \, V_1 f^1(\theta_1) = \infty \right \}, \quad E_2 = \left \{\theta_2 \in [0,2\pi) \, : \, V_1 f^2(\theta_2) = \infty \right \}.$$
Let $F = E \cup (E_1 \times \mathbb{T}) \cup (\mathbb{T} \cup E_2)$. Then $C(F) = 0$, by three applications of Theorem~\ref{thm:beurling}. Suppose now that $\theta \notin F$, and for $r, r' \in [0,1)^2$, write by analyticity
\begin{multline*}
f_r(\theta) - f_{r'}(\theta) = \int_0^{r_1} \int_0^{r_2} \partial_\rho f_\rho(\theta) \, d\rho - \int_0^{r_1'} \int_0^{r_2'} \partial_\rho f_\rho(\theta) \, d\rho \\ + \int_{r_1'}^{r_1} \partial_{\rho_1} f^1_{\rho_1}(\theta_1) \, d\rho_1 + \int_{r_2'}^{r_2} \partial_{\rho_2} f^2_{\rho_2}(\theta_2) \, d\rho_2.
\end{multline*}
Thus
\begin{multline*}
\|f_r(\theta) - f_{r'}(\theta)\|_{\mathcal{H}} \leq \int_{\min(r_1, r_1')}^1 \int_0^1 \|\partial_\rho f_\rho(\theta)\|_{\mathcal{H}} \, d\rho + \int_0^1 \int_{\min(r_2, r_2')}^1 \|\partial_\rho f_\rho(\theta)\|_{\mathcal{H}} \, d\rho \\
+ \int_{\min(r_1, r_1')}^{1} \|\partial_{\rho_1} f^1_{\rho_1}(\theta_1)\|_{\mathcal{H}} \, d\rho_1 + \int_{\min(r_2, r_2')}^{1} \|\partial_{\rho_2} f^2_{\rho_2}(\theta_2)\|_{\mathcal{H}} \, d\rho_2.
\end{multline*}
Since $V_2f(\theta)$, $V_1f^1(\theta_1)$, and $V_1 f^2(\theta_2)$ are all finite, it follows that 
$$\|f_r(\theta) - f_{r'}(\theta)\|_{\mathcal{H}} \to 0, \quad r, r' \to (1,1).$$
Hence $f^\ast(\theta) = \lim_{r\to(1,1)} f_r(\theta)$ exists, for every $\theta$ outside the capacity zero set $F$. Letting $r' = 0$ in the estimate also shows that $\|f_r(\theta)\|_{\mathcal{H}}$ is uniformly bounded in $r$. 

We postpone the proof that $f^\ast(\theta)$ coincides with the sum $f(\theta)$ quasi-everywhere to the proof of Theorem~\ref{thm:rectconv}.
\end{proof}
For $n=1$ and $\mathcal{H} = \mathbb{C}$, a series $f \in \mathcal{D}(\mathbb{D})$ is summable at $\theta \in [0,2\pi)$ if and only if it is Abel summable at $\theta$. This is sometimes known as Fej\'er's Tauberian theorem. Thus, in this case Theorem~\ref{thm:beurling} immediately implies Theorem~\ref{thm:rectconv}. To prove Theorem~\ref{thm:rectconv} for $n \geq 2$, we begin by stating a vector-valued version of Fej\'er's theorem.
\begin{lemma} \label{lem:fejer}
For $N \in \mathbb{N}$ and $\theta \in [0, 2\pi)$, define $S_{N, \theta}^{\mathcal{H}}, P_{N, \theta}^{\mathcal{H}} \colon \mathcal{D}(\mathbb{D}, \mathcal{H}) \to \mathcal{H}$ by
$$S_{N, \theta}^{\mathcal{H}} f = \sum_{k=0}^N a_k e^{ik\theta}, \; P_{N, \theta}^{\mathcal{H}} f = f_{1-1/N}(\theta), \quad f \in \mathcal{D}(\mathbb{D}, \mathcal{H}).$$
Then there is an absolute constant $C > 0$ such that
$$\|S_{N, \theta}^{\mathcal{H}} f - P_{N, \theta}^{\mathcal{H}} f\|_{\mathcal{H}} \leq C \|f\|_{\mathcal{D}(\mathbb{D}, \mathcal{H})}.$$
Moreover, for every fixed $f$ we have that
$$S_{N, \theta}^{\mathcal{H}} f - P_{N, \theta}^{\mathcal{H}} f \to 0, \quad N\to\infty,$$
uniformly in $\theta$.
\end{lemma}
\begin{proof}
Let $r = 1-1/N$, and note that $1-r^k \leq k/N$, to see that
$$\|S_{N, \theta}^{\mathcal{H}} f - P_{N, \theta}^{\mathcal{H}} f\|_{\mathcal{H}} \leq \frac{1}{N}\sum_{k=1}^N k\|a_k\|_{\mathcal{H}} + \sum_{k=N}^\infty \|a_k\|_{\mathcal{H}} r^k.$$
For $M \leq N$, we estimate
$$\frac{1}{N}\sum_{k=1}^N k\|a_k\|_{\mathcal{H}} \leq \frac{1}{N}\sum_{k=1}^M k\|a_k\|_{\mathcal{H}} + \frac{1}{N}\left( \sum_{k=M}^N k\|a_k\|_{\mathcal{H}}^2 \right)^{1/2}\left( \sum_{k=M}^N k\right)^{1/2}.$$
By first choosing $M$ large, and then $N$, we see that $\frac{1}{N}\sum_{k=1}^N k\|a_k\|_{\mathcal{H}} \to 0$ as $N \to \infty$. For the second term we have that
$$
\sum_{k=N}^\infty \|a_k\|_{\mathcal{H}} r^k \leq \frac{1}{\sqrt{N}}\left( \sum_{k=N}^\infty k \|a_k\|_{\mathcal{H}}^2 \right)^{1/2}\left( \sum_{k=N}^\infty r^{2k}\right)^{1/2},
$$
and thus this term also tends to $0$ as $N \to \infty$. This second estimate, together with the first estimate for $M=0$, also shows the uniform bound of the operator norm of $S_{N, \theta}^{\mathcal{H}} - P_{N, \theta}^{\mathcal{H}}$.
\end{proof}
In the proof of Theorem~\ref{thm:rectconv} we will consider tensors of the operators $S_{N, \theta}$ and $P_{N, \theta}$, interpreted in the obvious way. For instance, if $N \in \mathbb{N}^n$, $\theta \in [0,2\pi)^n$, and $f \in \mathcal{D}(\mathbb{D}^n, \mathcal{H})$, then
$$(S_{N_1, \theta_1} \otimes \cdots \otimes S_{N_n, \theta_n}) f =  \sum_{\alpha_1= 0}^{N_1} \cdots \sum_{\alpha_n= 0}^{N_n} a_\alpha e^{i(\alpha, \theta)},$$
and
\begin{align*}
(P_{N_1, \theta_1} \otimes \cdots \otimes P_{N_n, \theta_n}) f &= f_{(1-1/N_1, \ldots, 1-1/N_n)}(\theta) \\ &=  \sum_{\alpha_1= 0}^{\infty} \cdots \sum_{\alpha_n= 0}^{\infty} a_\alpha (1-1/N_1)^{\alpha_1} \cdots (1-1/N_n)^{\alpha_n}e^{i(\alpha, \theta)}.
\end{align*}
Similarly, we consider mixed tensor products, such as
$$(S_{N_1, \theta_1} \otimes P_{N_2, \theta_2})f = \sum_{\alpha_1= 0}^{N_1} \sum_{\alpha_2= 0}^{\infty} a_{\alpha_1, \alpha_2} (1-1/N_2)^{\alpha_2} e^{i(\alpha, \theta)}.$$

\begin{proof}[Proof of Theorem~\ref{thm:rectconv}]
We will deduce the result from Theorem~\ref{thm:beurling}, Lemma~\ref{lem:fejer}, and an inductive procedure which exploits the fact that $$\mathcal{D}(\mathbb{D}^{n}, \mathcal{H})  = \mathcal{D}(\mathbb{D}^{n-1}, \mathcal{D}(\mathbb{D}, \mathcal{H})).$$
We already know that Theorem~\ref{thm:rectconv} is true for $n=1$, precisely by Theorem~\ref{thm:beurling} and Lemma~\ref{lem:fejer}.

Thus we first consider the case $n=2$. By Corollary~\ref{cor:conv}, there is a Borel set $E \subset \mathbb{T}^2$ such that $C(\mathbb{T}^2 \setminus E) = 0$, and  for every $\theta = (\theta_1, \theta_2) \in E$ we have that $(P_{N_1, \theta_1} \otimes P_{N_2, \theta_2})f$ is uniformly bounded in $N_1, N_2$ and convergent to $f^\ast(\theta)$ as $N_1, N_2 \to \infty$. To prove the theorem, it is thus sufficient to provide a set $F \subset E$ such that $C(E \setminus F) = 0$ and such that for every $\theta \in F$ it holds that
\begin{equation} \label{eq:tensor1}
\lim_{N_1,N_2 \to \infty}\|(S_{N_1, \theta_1} \otimes S_{N_2, \theta_2} - P_{N_1, \theta_1} \otimes P_{N_2, \theta_2})f\|_{\mathcal{H}} = 0,
\end{equation}
and
\begin{equation} \label{eq:tensor2}
\sup_{N_1,N_2}\|(S_{N_1, \theta_1} \otimes S_{N_2, \theta_2} - P_{N_1, \theta_1} \otimes P_{N_2, \theta_2})f\|_{\mathcal{H}} < \infty.
\end{equation}
Constructing such a set $F$ of course also proves that $f^\ast(\theta) = f(\theta)$ quasi-everywhere, as claimed in Corollary~\ref{cor:conv}.

We write
\begin{multline*}
(S_{N_1, \theta_1} \otimes S_{N_2, \theta_2} - P_{N_1, \theta_1} \otimes P_{N_2, \theta_2})f \\
= ((S_{N_1, \theta_1} - P_{N_1,\theta_1}) \otimes S_{N_2, \theta_2})f + (P_{N_1, \theta_1} \otimes (S_{N_2, \theta_2} - P_{N_2, \theta_2}))f.
\end{multline*}
Now, by the $n=1$ case of the theorem, applied to $f \in \mathcal{D}(\mathbb{D}, \mathcal{D}(\mathbb{D}, \mathcal{H}))$, there is a set $G_2 \subset \mathbb{T}$ such that $C(\mathbb{T} \setminus G_2) = 0$, and such that for every $\theta_2 \in G_2$ we have the existence of
\begin{equation} \label{eq:hdef}
h_{\theta_2} := \lim_{N_2 \to \infty} S_{N_2, \theta_2}^{\mathcal{D}(\mathbb{D}, \mathcal{H})} f \in \mathcal{D}(\mathbb{D}, \mathcal{H}).
\end{equation}
Next, for $\theta_2 \in G_2$, note that
\begin{multline*}
((S_{N_1, \theta_1} - P_{N_1,\theta_1}) \otimes S_{N_2, \theta_2})f = (S_{N_1, \theta_1}^{\mathcal{H}} - P_{N_1,\theta_1}^{\mathcal{H}})S_{N_2, \theta_2}^{\mathcal{D}(\mathbb{D}, \mathcal{H})} f \\
=(S_{N_1, \theta_1}^{\mathcal{H}} - P_{N_1,\theta_1}^{\mathcal{H}})(S_{N_2, \theta_2}^{\mathcal{D}(\mathbb{D}, \mathcal{H})} f - h_{\theta_2}) + (S_{N_1, \theta_1}^{\mathcal{H}} - P_{N_1,\theta_1}^{\mathcal{H}}) h_{\theta_2}.
\end{multline*}
Thus, by Lemma~\ref{lem:fejer} and \eqref{eq:hdef} it follows that, for any fixed $(\theta_1, \theta_2) \in \mathbb{T} \times G_2$, the term $((S_{N_1, \theta_1} - P_{N_1,\theta_1}) \otimes S_{N_2, \theta_2})f$ is uniformly bounded in $N_1, N_2$ and tends to $0$ as $N_1, N_2 \to \infty$. 

By a very similar argument (after reordering the variables $\theta_1$ and $\theta_2$), there is a set $G_1 \subset \mathbb{T}$ such that $C(\mathbb{T} \setminus G_1) = 0$, and such that for every $\theta_1 \in G_1$ and $\theta_2 \in \mathbb{T}$, the term  $(P_{N_1, \theta_1} \otimes (S_{N_2, \theta_2} - P_{N_2, \theta_2}))f$ is uniformly bounded in $N_1, N_2$ and tends to zero as $N_1, N_2 \to \infty$. Thus the proof for $n=2$ is finished by letting
$$F = E \cap (G_1 \times \mathbb{T}) \cap (\mathbb{T} \times G_2).$$
Note that in the course of the proof we have also established that $(P_{N_1, \theta_1} \otimes S_{N_2, \theta_2}) f$ is uniformly bounded in $N_1, N_2$ and converges to $f^\ast(\theta)$ as $N_1, N_2 \to \infty$, for $\theta \in F$.

For $n=3$, Corollary~\ref{cor:conv} gives us a set $E \subset \mathbb{T}^3$ such that $C(\mathbb{T}^3 \setminus E) = 0$ and on which $(P_{N_1, \theta_1} \otimes P_{N_2, \theta_2} \otimes P_{N_3, \theta_3})f$ converges and is uniformly bounded. We then write
\begin{multline*}
(S_{N_1, \theta_1} \otimes S_{N_2, \theta_2}  \otimes S_{N_3, \theta_3}  - P_{N_1, \theta_1} \otimes P_{N_2, \theta_2} \otimes P_{N_3, \theta_3})f = \\
((S_{N_1, \theta_1} - P_{N_1,\theta_1}) \otimes S_{N_2, \theta_2} \otimes S_{N_3, \theta_3})f + (P_{N_1, \theta_1} \otimes (S_{N_2, \theta_2} - P_{N_2, \theta_2} ) \otimes S_{N_3, \theta_3})f \\ +  (P_{N_1, \theta_1} \otimes  P_{N_2, \theta_2} \otimes ( S_{N_3, \theta_3} - P_{N_3, \theta_3})f.
\end{multline*}
Now we apply the $n=2$ case of the theorem, together with the remark at the end of its proof, three separate times to $f \in \mathcal{D}(\mathbb{D}^2, \mathcal{D}(\mathbb{D}, \mathcal{H}))$. Arguing with Lemma~\ref{lem:fejer} as before, this produces three sets $H_1, H_2, H_3 \subset \mathbb{T}^3$ such that $C(\mathbb{T}^3 \setminus H_j) = 0$, and such that, for $\theta \in H_j$, the $j$:th term is uniformly bounded in $N_1, N_2, N_3$ and converges  to zero as $N_1,N_2, N_3 \to\infty$. Thus $(S_{N_1, \theta_1} \otimes S_{N_2, \theta_2}  \otimes S_{N_3, \theta_3})f$ is uniformly bounded and converges as $N_1, N_2, N_3 \to \infty$, for $\theta \in E \cap H_1 \cap H_2 \cap H_3$. Furthermore, the same is true of $(P_{N_1, \theta_1} \otimes S_{N_2, \theta_2}  \otimes S_{N_3, \theta_3})f$ and $(P_{N_1, \theta_1} \otimes P_{N_2, \theta_2}  \otimes S_{N_3, \theta_3})f$.

It is now clear how the construction extends by induction to $n \geq 4$.
\end{proof}
To conclude this section, we consider Theorem~\ref{thm:strongdiff}.
One potential approach is to use a capacitary weak type inequality for the strong maximal function, or for the iterate of one-variable maximal functions. See \cite[Theorem~6.2.1]{Hedberg} for the one-parameter case. Instead of pursuing this, we will give a different argument which directly connects Theorem~\ref{thm:strongdiff} with Theorem~\ref{thm:rectconv}.
\begin{proof}[Proof of Theorem~\ref{thm:strongdiff}]
Note first that
\begin{equation} \label{eq:riemannsum}
F_h(\theta) = \sum_{\alpha \in \mathbb{N}^n} a_\alpha \frac{\sin(\alpha_1h)}{\alpha_1 h} \cdots \frac{\sin(\alpha_n h)}{\alpha_n h} e^{i(\alpha, \theta)}.
\end{equation}
This is obviously true for polynomials, and for all $f \in \mathcal{D}(\mathbb{D}^n, \mathcal{H})$ by continuity. For this last statement, note that, with continuous dependence on $f$, the values $f(\theta)$ are square-integrable on $\mathbb{T}^n$, and the right-hand side of \eqref{eq:riemannsum} is absolutely convergent.

The argument is now very similar to the proof of Theorem~\ref{thm:rectconv}. First we consider the case $n=1$, letting
$$R_{h, \theta}^\mathcal{H} f = \sum_{k= 0}^\infty a_{k}  \frac{\sin(k h)}{k h} e^{i k \theta}, \quad f \in \mathcal{D}(\mathbb{D}, \mathcal{H}),$$
for $\theta \in [0, 2\pi)$ and $h \in (0,1)$. Let $1 \leq N \in \mathbb{N}$ be such that $\frac{1}{N+1} \leq h < \frac{1}{N}$, and let $M \leq N$. Then
\begin{multline*}
\|R_{h, \theta}^\mathcal{H}f - S_{N, \theta}^\mathcal{H}f\|_\mathcal{H} \lesssim \sum_{k=1}^N \|a_k\|_{\mathcal{H}} (k h)^2  +  \sum_{k=N}^\infty \frac{\|a_k\|_{\mathcal{H}}}{k h} \leq 
\sum_{k=1}^M \|a_k\|_{\mathcal{H}} (k h)^2  \\ +  \left( \sum_{k=M}^N k \|a_k\|_{\mathcal{H}}^2 \right)^{\frac12}\left(h^4 \sum_{k=M}^N k^3 \right)^{\frac12} + \left( \sum_{k=N}^\infty k \|a_k\|_{\mathcal{H}}^2 \right)^{\frac12}\left( \frac{1}{h^2} \sum_{k=N}^\infty \frac{1}{k^3} \right)^{\frac12}.
\end{multline*}
By this estimate, $R_{h, \theta}^\mathcal{H} - S_{N, \theta}^\mathcal{H} \colon \mathcal{D}(\mathbb{D}, \mathcal{H}) \to \mathcal{H}$ is uniformly bounded in $N$ and converges pointwise to $0$ as $N \to \infty$, as long as $\frac{1}{N+1} \leq h < \frac{1}{N}$. Thus Theorem~\ref{thm:rectconv} implies Theorem~\ref{thm:strongdiff} in the case that $n=1$.

For $n \geq 2$ we proceed precisely as in the proof of Theorem~\ref{thm:rectconv}. For instance, for $n=2$ we write
\begin{multline*}
(S_{N_1, \theta_1} \otimes S_{N_2, \theta_2} - R_{h_1, \theta_1} \otimes R_{h_2, \theta_2})f \\
= ((S_{N_1, \theta_1} - R_{h_1,\theta_1}) \otimes S_{N_2, \theta_2})f + (R_{h_1, \theta_1} \otimes (S_{N_2, \theta_2} - R_{h_2, \theta_2}))f,
\end{multline*}
where $N = (N_1, N_2)$ is related to $h = (h_1, h_2)$ by the facts that $\frac{1}{N_j+1} \leq h_j < \frac{1}{N_j}$, $j=1,2$. The rest of the proof is essentially repetition.
\end{proof}
\section{Sharpness of results}
To prove Theorem~\ref{thm:sharp} in the multi-parameter setting, we adapt a one-variable construction of Carleson which is well described for example in \cite[Theorem~3.4.1]{DirichletBook}.
\begin{proof}[Proof of Theorem~\ref{thm:sharp}]
Since $C(\cdot)$ is outer and $C(E) = 0$, we may choose a sequence $G_1 \supset G_2 \supset G_3 \supset \cdots$ of open sets such that $E \subset G_j$, for all $j$, and 
$$\sum_{j=1}^\infty C(G_j)^{1/2} < \infty.$$
Since $E$ is compact, we may additionally assume that $\overline{G}_{j+1} \subset G_j$ for every $j$. Letting $F_j = \overline{G}_j$, we thus have a decreasing sequence $F_1 \supset F_2 \supset F_3 \supset \cdots$ of compact sets containing $E$, such that
\begin{equation} \label{eq:capsum}
\sum_{j=1}^\infty C(F_j)^{1/2} < \infty.
\end{equation}

Let $\mu_{F_j}$ be the equilibrium measure of $F_j$, and define $f_j \in \mathcal{D}(\mathbb{D}^n)$ by the relationship
$$f_j(z) = \int_{\mathbb{T}^n} \left(C + \log \frac{1}{1-z_1e^{-i\psi_1}}\right) \cdots \left(C + \log \frac{1}{1-z_n e^{-i\psi_n}}\right) \, d\mu_{F_j}(\psi),$$
for  $z \in \mathbb{D}^n$.
Let
$$G(\psi) = \left(C + \log \frac{1}{1-e^{-i\psi_1}}\right) \cdots \left(C + \log \frac{1}{1- e^{-i\psi_n}}\right), \quad \psi \in [0,2\pi)^n.$$
It is key to the proof that if we choose $C > 0$ sufficiently large, then
\begin{equation} \label{eq:recomp}
 \mre G(\psi) \approx H(\psi).
\end{equation}
In particular,
$$\mre \left(C + \log \frac{1}{1-z_1e^{-i\psi_1}}\right) \cdots \left(C + \log \frac{1}{1-z_n e^{-i\psi_n}}\right) \geq 0, $$
for $z \in \mathbb{D}^n$ and $\psi \in [0,2\pi)^n$, since the left-hand side is the Poisson integral of $\mre G(\psi - \cdot)$.
Therefore we fix $C$ as a constant such that \eqref{eq:recomp} holds. The choice of $C$ only depends on $n$. 

 With $\widehat{\mu}_{F_j}(\alpha) = \int_{\mathbb{T}^n} e^{-i(\alpha, \theta)} \, d\mu_{F_j}(\theta)$, we then have that
\begin{align*}
\|B\mu_{F_j}\|_{L^2(\mathbb{T}^n)}^2 &= \int_{\mathbb{T}^n} \int_{\mathbb{T}^n} H(\theta - \psi) \, d\mu_{F_j}(\psi) \, d\mu_{F_j}(\theta) \\&\approx \mre  \int_{\mathbb{T}^n} \int_{\mathbb{T}^n} G(\theta - \psi)  \, d\mu_{F_j}(\psi) \, d\mu_{F_j}(\theta) 
\approx \sum_{\alpha \in \mathbb{N}^n} \frac{|\widehat{\mu}_{F_j}(\alpha)|^2}{(\alpha_1+1)\cdots (\alpha_n+1)},
\end{align*}
where the last step follows by a computation with coefficients (including a straightforward approximation argument). A computation with Fourier coefficients also yields that
$$\|f_j\|_{\mathcal{D}(\mathbb{D}^n)}^2 \approx \sum_{\alpha \in \mathbb{N}^n} \frac{|\widehat{\mu}_{F_j}(\alpha)|^2}{(\alpha_1+1)\cdots (\alpha_n+1)} \approx \|B\mu_{F_j}\|_{L^2(\mathbb{T}^n)}^2 = C(F_j).$$
In view of \eqref{eq:capsum} we may therefore define the function
$$f = \sum_{j=1}^\infty f_j \in \mathcal{D}(\mathbb{D}^n).$$
We will demonstrate that $\lim_{z \to \zeta} \mre f(z) = \infty$, for every $\zeta \in E$. 

Since $\mre f_j$ is $n$-harmonic and non-negative, there is by Lemma~\ref{lem:poissonrep} a measure $d\mu_j = g_j \, d\theta + d\sigma_j$ such that $0 \leq g_j \in L^1(\mathbb{T}^n)$, $\sigma_j \geq 0$ is singular, and $\mre f_j(z) = P\mu_j(z)$ for $z \in \mathbb{D}^n$. By Corollary~\ref{cor:conv} the limit $\lim_{t\to1} \mre f_j(t e^{i\theta_1}, \ldots, te^{i\theta_n})$ exists for quasi-every, and thus almost every, $\theta \in [0,2\pi)^n$. Furthermore, by Fatou's lemma and the properties of an equilibrium measure, we have that
$$\lim_{t\to1} \mre f_j(t e^{i\theta_1}, \ldots, te^{i\theta_n}) \geq \mre \int_{\mathbb{T}^n} G(\psi-\theta) \, d\mu_{F_j}(\psi) \approx \int_{\mathbb{T}^n} H(\theta - \psi) \, d\mu_{F_j}(\psi) \geq 1$$
for quasi-every $\theta \in F_j$. On the other hand, by Lemma~\ref{lem:poissonrep}, we have that 
$$\lim_{t\to1} \mre f_j(t e^{i\theta_1}, \ldots, te^{i\theta_n}) = g_j(\theta)$$ for almost every $\theta \in [0,2\pi)^n$. We conclude that there is a constant $c > 0$, independent of $j$, such that $g_j(\theta) \geq c$ for almost every $\theta$ in the open set $G_j \supset E$. 

Note that $P[d\psi] \equiv 1$. Given $\vartheta \in [0,2\pi)^n$ and $\varepsilon > 0$, let 
$$I_{\vartheta, \varepsilon} = \{\psi \, : \, \max_j |e^{i\vartheta_j} - e^{i\psi_j}| < \varepsilon\}.$$
Then
$$P[\chi_{\mathbb{T}^n \setminus I_{\vartheta,\varepsilon}}  d\psi](z) \leq \sum_{j=1}^n \int_{|e^{i\vartheta_j} - e^{i\psi_j}| \geq \varepsilon} P_{r_j}(\theta_j - \psi_j) \, d\psi_j, $$
where $z = (r_1e^{i\theta_1},\ldots, r_ne^{i\theta_n}) \in \mathbb{D}^n$. Thus 
$$P[\chi_{I_{\vartheta,\varepsilon}}  d\psi](z) = 1 - P[\chi_{\mathbb{T}^n \setminus I_{\vartheta,\varepsilon}}  d\psi](z) \to 1,$$
as $z \to (e^{i\vartheta_1}, \ldots, e^{i\vartheta_n})$. We conclude, for $\zeta \in E \subset G_j$, that
$$\varliminf_{z \to \zeta} \mre f_j(z) \geq \varliminf_{z \to \zeta} cP[\chi_{G_j} d\psi](z) \geq c.$$
Hence, for $\zeta \in E$, 
\begin{equation*}
\varliminf_{z \to \zeta} \mre f(z) \geq \sum_{j=1}^\infty \varliminf_{z\to \zeta} \mre f_j(z) = \infty. \qedhere
\end{equation*}
\end{proof}

\begin{proof}[Proof of Theorem~\ref{thm:rectsharp}]
This follows at once from Theorem~\ref{thm:sharp} and the fact that a multiple series which converges in the sense of Pringsheim has uniformly bounded and convergent iterated Abel means. This can be deduced from the standard proof of Abel's theorem, see for example \cite{BromHardy}. 

For completeness, let us sketch a proof for our setting, in the case that $n=2$. Hence assume that $f \in \mathcal{D}(\mathbb{D}^2, \mathcal{H})$ and that $f(\theta)$ is Pringsheim convergent. Without loss of generality, we may suppose that $f(\theta) = 0$. Then the summation by parts formula
\begin{equation} \label{eq:sumbyparts}
f_{r_1, r_2}(\theta) = (1-r_1)(1-r_2)\sum_{N \in \mathbb{N}^2}  r_1^{N_1} r_2^{N_2} (S_{N_1, \theta_1} \otimes S_{N_2, \theta_2})f 
\end{equation}
is clearly justified, since both sides are absolutely convergent. Indeed, by assumption,
$$C = \sup_{N \in \mathbb{N}^2} \|(S_{N_1, \theta_1} \otimes S_{N_2, \theta_2})f\|_{\mathcal{H}} < \infty.$$
The formula \eqref{eq:sumbyparts} immediately shows that $\sup_{r_1, r_2} \|f_{r_1, r_2}(\theta)\|_{\mathcal{H}} < \infty$. Furthermore, for any $M \in \mathbb{N}^2$, splitting the summation into the four index regions 
\begin{align*}
N_1 &\leq M_1, N_2 \leq M_2; \quad N_1 \leq M_1, N_2 > M_2; \\
N_1 &> M_1, N_2 \leq M_2; \quad N_1 > M_1, N_2 > M_2,
\end{align*}
 yields the estimate
\begin{multline*}
\|f_{r_1, r_2}(\theta)\|_{\mathcal{H}} \leq C(M_1+1)(M_2+1) (1-r_1)(1-r_2) + C(M_1+1)(1-r_1) \\ + C(M_2+1)(1-r_2) + \sup_{N_1 > M_1, N_2 > M_2} \|(S_{N_1, \theta_1} \otimes S_{N_2, \theta_2})f\|_{\mathcal{H}}.
\end{multline*}
This evidently implies that $f_{r_1, r_2}(\theta) \to 0$ as $(r_1, r_2) \to (1,1)$.
\end{proof}

\bibliographystyle{amsplain} 
\bibliography{multfour} 
\end{document}